\documentclass[11pt,a4paper]{amsart}
\usepackage{amssymb,amsmath,epsfig,graphics,mathrsfs}

\usepackage{amsthm,epsfig,esint}
\usepackage{ulem}

\usepackage{fancyhdr}
\usepackage{hyperref}
\hypersetup{
 colorlinks   = true,
 urlcolor     = blue,
 linkcolor    = blue,
 citecolor   = red ,
 bookmarksopen=true
}

\usepackage{amsmath}
\usepackage{amsfonts}
\usepackage{amssymb}
\usepackage{amsthm}
\usepackage{epsfig,graphics,mathrsfs}
\usepackage{graphicx}

\usepackage[usenames, dvipsnames]{color} 

\usepackage{hyperref}

\usepackage{ulem}

\theoremstyle{plain}
\newtheorem{thrm}{Theorem}[section]
\newtheorem*{thrm*}{Theorem}
\newtheorem{lemma}[thrm]{Lemma}
\newtheorem{prop}[thrm]{Proposition}
\newtheorem{cor}[thrm]{Corollary}
\theoremstyle{definition}
\newtheorem{dfn}[thrm]{Definition}
\theoremstyle{remark}
\newtheorem{rmrk}[thrm]{Remark}

\theoremstyle{example}
\numberwithin{equation}{section}
\usepackage{color}

\begin{document}

\newcommand{\ddelta}{\delta}

\newcommand{\tx}{\tilde x}
\newcommand{\R}{\mathbb R}
\newcommand{\N}{\mathbb N}
\newcommand{\C}{\mathbb C}
\newcommand{\lie}{\mathcal G}
\newcommand{\hN}{\mathcal N}
\newcommand{\D}{\mathcal D}
\newcommand{\A}{\mathcal A}
\newcommand{\B}{\mathcal B}
\newcommand{\sL}{\mathcal L}
\newcommand{\sLi}{\mathcal L_{\infty}}

\newcommand{\G}{\Gamma}
\newcommand{\x}{\xi}

\newcommand{\eps}{\epsilon}
\newcommand{\al}{\alpha}
\newcommand{\be}{\beta}
\newcommand{\p}{\partial}  
\newcommand{\lig}{\mathfrak}

\def\dist{\mathop{\varrho}\nolimits}

\newcommand{\BCH}{\operatorname{BCH}\nolimits}
\newcommand{\Lip}{\operatorname{Lip}\nolimits}
\newcommand{\Hol}{C}                             
\newcommand{\lip}{\operatorname{lip}\nolimits}
\newcommand{\capQ}{\operatorname{Cap}\nolimits_Q}
\newcommand{\pCap}{\operatorname{Cap}\nolimits_p}
\newcommand{\Om}{\Omega}
\newcommand{\om}{\omega}
\newcommand{\half}{\frac{1}{2}}
\newcommand{\e}{\varepsilon}
\newcommand{\vn}{\vec{n}}
\newcommand{\X}{\Xi}
\newcommand{\tLip}{\tilde  Lip}

\newcommand{\Span}{\operatorname{span}}

\newcommand{\ad}{\operatorname{ad}}
\newcommand{\Hm}{\mathbb H^m}
\newcommand{\Hn}{\mathbb H^n}
\newcommand{\Hone}{\mathbb H^1}
\newcommand{\Lie}{\mathfrak}
\newcommand{\Layer}{V}
\newcommand{\hgrad}{\nabla_{\!H}}
\newcommand{\im}{\textbf{i}}
\newcommand{\nz}{\nabla_0}
\newcommand{\s}{\sigma}
\newcommand{\se}{\sigma_\e}

\newcommand{\ued}{u^{\e,\ddelta}}
\newcommand{\ueds}{u^{\e,\ddelta,\sigma}}
\newcommand{\tnabla}{\tilde{\nabla}}

\newcommand{\bx}{\bar x}
\newcommand{\by}{\bar y}
\newcommand{\bt}{\bar t}
\newcommand{\bs}{\bar s}
\newcommand{\bz}{\bar z}
\newcommand{\btau}{\bar \tau}

\newcommand{\LC}{\mbox{\boldmath $\nabla$}}
\newcommand{\Ne}{\mbox{\boldmath $n $}}
\newcommand{\nuo}{\mbox{\boldmath $n^0$}}
\newcommand{\nuu}{\mbox{\boldmath $n^1$}}
\newcommand{\nue}{\mbox{\boldmath $n $}}
\newcommand{\nuek}{\mbox{\boldmath $n^{\e_k}$}}
\newcommand{\dse}{\nabla^{H\Su, \e}}
\newcommand{\dso}{\nabla^{H\Su, 0}}
\newcommand{\tX}{\tilde X}

\newcommand\red{\textcolor{red}}
\newcommand\green{\textcolor{green}}

\newcommand{\Xie}{X psilon_i}
\newcommand{\Xje}{X psilon_j}
\newcommand{\Su}{\mathcal S}
\newcommand{\F}{\mathcal F}

\def\pa{\partial}
\def\Id{{\rm Id}\,}
\def\loc{{\rm loc}}
\def\div{{\rm div}\,}
\def\reg{{\rm reg}\,}
\def\regn{{\rm regn}\,}
\def\Cap{{\rm Cap}}

\title [Parabolic $p$-Laplacian in the Heisenberg group]{ Lipschitz regularity for solutions of the parabolic $p$-Laplacian 
 in the Heisenberg group}

\author[L. Capogna]{L. Capogna}
\address{Luca Capogna\\Department of Mathematics and Statistics, Smith College, Northampton, MA 01060\\
}\email{lcapogna@smith.edu}

\author[G.  Citti]{G.  Citti}\address{Dipartimento di Matematica, Piazza Porta S. Donato 5, Universit\`a di Bologna, 
40126 Bologna, Italy}\email{giovanna.citti@unibo.it}
%
%
\author{Xiao Zhong}\address{Department of Mathematics and Statistics, University of Helsinki, 00014 University of Helsinki, Finland}
\email{xiao.x.zhong@helsinki.fi}


\thanks{LC was partially supported by NSF award  DMS1955992}
\thanks{GC was partially funded by  
Horizon 2020 Project ref. 777822: GHAIA}
\thanks{XZ was supported by the Academy of Finland, project \#308759.}

\begin{abstract}
We prove local Lipschitz regularity for  weak solutions to a class of  degenerate parabolic PDEs modeled on the parabolic $p$-Laplacian $$\p_t u= \sum_{i=1}^{2n} X_i (|\nabla_0 u|^{p-2} X_i u),$$ in a cylinder $\Omega\times \R^+$, where $\Omega$ is domain in the Heisenberg group $\Hn$, and  $2\le p \le 4$. The result continues to hold  in the more general setting of contact sub-Riemannian manifolds.
 \end{abstract}
\maketitle

\maketitle


\section{Introduction}

In this paper we establish the  local Lipschitz regularity of weak solutions of a certain class of quasilinear, degenerate parabolic equations in the Heisenberg group $\Hn$, or more in general in contact subRiemannian manifolds. In particular we extend to the non-stationary setting the early work \cite{Manfredi-Mingione, Mingione-Zatorska-Zhong}, by introducing  a new, more elegant approach to the regularity problem.

In a cylinder $Q=\Omega\times (0,T)$, where $\Omega\subset\Hn$ is an open set and $T>0$, we consider the equation
\begin{equation}\label{maineq-zero}
\p_t u = \sum_{i=1}^{2n} X_i A_i(x, \nabla_0 u) \quad \quad  \text{ in  \quad }Q=\Omega\times (0,T),
\end{equation}
modeled on the  parabolic $p$-Laplacian 
\begin{equation}\label{plap-para} \p_t u= \sum_{i=1}^{2n} X_i \bigg(|\nabla_0 u|^{p-2} X_i u\bigg),\end{equation}
where  $ 2\le p \le 4$ and $X_1,...,X_{2n}$ denote the horizontal left invariant frame in $\Hn$. 
In a previous study \cite{CCG}, Garofalo and the first two listed authors have extended techniques originally introduced by the third listed author \cite{Zhong} to establish $C^\infty$ smoothness for weak solutions to \eqref{maineq} in the range $2\le p <\infty$ under some additional non-degeneracy hypothesis. 
 In the present paper we show that in the restricted range of the nonlinearity $2\le p \le 4$, and without the additional non-degeneracy assumptions, one can obtain Lipschitz regularity of weak solutions.
 
We indicate with $x=(x_1,...,x_{2n},x_{2n+1})$ the variable point in $\Hn$. We will occasionally denote the variable $x_{2n+1}$ in the center of the group with the letter $z$. Consequently, we will indicate with $\p_i$ partial differentiation with respect to the variable $x_i$, $i=1,...,2n$, and use the notation $Z = \p_z$ for the partial derivative $\p_{x_{2n+1}}$. The notation $\nabla_0 u=\sum_{i=1}^{2n} X_i u X_i \cong (X_1 u,...,X_{2n} u)$ denotes the so-called {\it horizontal gradient} of the function $u$, where
$$X_i = \partial_i - \frac{x_{n+i}}2 \partial_{z},\ \ \ \ \ \   
X_{n+i} = \partial_{n+i} + \frac{x_{i}}{2} \partial_{z},\ \ \ \  i=1,...,n.$$
As it is well-known, the $2n+1$ vector fields $X_1,...,X_{2n}, Z$ are connected by the following commutation relation: for every couple of index $i, j$, if $j=i+n$, then $[X_i,X_j] = Z$; all other commutators being trivial.  

\noindent{\bf Structural assumptions.}  The relevant  assumptions on the vector-valued function $$(x,\xi)\to A(x,\xi)=(A_1(x,\xi),...,A_{2n}(x,\xi))$$ are that
 there exist $2\le p \le 4$,  and  $0<\lambda'\le \Lambda' <\infty$ such that for a.e. $x\in\Omega, \xi\in \R^{2n}$ and for all $\eta\in \R^{2n}$, one has
\begin{equation}\label{structure-zero}
\begin{cases}
 \lambda' |\xi|^{p-2} |\eta|^2 \le  \p_{\xi_j} A_i(x,\xi) \eta_i \eta_j \le \Lambda' |\xi|^{p-2} |\eta|^2,
 \\ 
|A_i(x,\xi)| +  |\p_{x_j} A_i(x,\xi)| \le  \Lambda' |\xi|^{p-1}.
\end{cases}
\end{equation}

Given an open set $\Omega\subset \Hn$, we indicate with $W^{1,p}(\Om)$ the Sobolev space associated with the $p$-energy $\mathscr E_{\Om,p}(u) = \frac 1p \int_{\Omega}|\nabla_0 u|^p$, i.e., the space of all functions $u\in L^p(\Omega)$ such that their distributional derivatives $X_i u$, $i=1,...,2n,$ are also in $L^p(\Omega)$. The corresponding norm
is $||u||^p_{W^{1,p}(\Omega)} = ||u||_{L^p(\Omega)}+||\nabla_0 u||_{L^p(\Omega)}.$ We 
will add the subscript {\it loc} for the local versions of such spaces, and denote by $W^{1,p}_0(\Om)$ the completion of $C^\infty_0(\Om)$ with respect to such norm.  A function $u\in L^p((0,T), W_{\loc}^{1,p}(\Om))$
is a weak solution of \eqref{maineq} in the cylinder $\Omega\times (0,T)$ if
\begin{equation}\label{weak}
\int_0^T \int_\Omega u \phi_t -  \sum_{i=1}^{2n} A_i(x,\nabla_0 u) X_i \phi  =0,
\end{equation}
for every $\phi\in C^{\infty}_0(Q)$.  Our main result is a Lipschitz regularity estimate for weak solutions, on parabolic cylinders $Q_{\mu,r}$ (see Definition \ref{cylinder}).

%
%
%
%
%
%
%
%

\begin{thrm}\label{main1}  
Let $A_i$ satisfy the structure conditions \eqref{structure-zero} and 
let $u\in L^p((0,T), W_{\loc}^{1,p}(\Om))$ be a weak solution of \eqref{maineq-zero}  in $Q=\Om\times (0,T)$.  If $2\le p \le 4$ then $|\nabla_0 u| \in L^\infty_{\loc}(Q)$ and $\partial_t u, Zu \in L^q_{\loc}(Q)$ for every $1\le q<\infty$. Moreover, one has
that  for any $Q_{\mu, 2r}\subset Q$,
\begin{equation}\label{desired}
\sup_{Q_{\mu, r}}  |\nabla_0u|
\le C  \max\Big( \Big(\frac{1}{\mu r^{N+2}}\int\int_{Q_{\mu,2r}} (\delta+|\nabla_0u|^2)^{\frac{p}{2}}\Big)^{\frac{1}{2}}, \mu^{\frac{p}{2(2-p)}}\Big),
\end{equation}
where $C=C(n,p,\lambda, \Lambda, r,\mu)>0$.  In the special case where there is no direct dependence on the space variable, i.e.  $A_i(x,\xi)=A_i(\xi)$, the parameters dependence  is more explicit, with
$C=C(n,p,\lambda, \Lambda) \mu^{\frac 1 2}>0$.
\end{thrm}

The boundedness of the gradient of solutions to equation (\ref{maineq-zero}) in the setting of Euclidean spaces is well-known \cite{DF, Ch, DB}. The above theorem shows that this is also the case in the setting of Heisenberg group for the range $2\le p\le 4$. We believe that it is true for all range $1<p<\infty$ if the solution is bounded as in the Euclidean setting. 
The proof of the main theorem is based on the Caccioppoli type estimate in Proposition \ref{cor1}, which is new in the subelliptic context. The Lipschitz regularity follows  then through a Moser type iteration, which we present in detail in Section \ref{lipschitz estimate}. The Caccioppoli estimate in Proposition \ref{cor1} is derived using two main ingredients: The first of these  is an approximation scheme, that allows us to invoke the regularity results from \cite{CCG}, thus dealing with smooth approximants that can be differentiated directly, without recurring to fractional difference quotients. One of the  original contributions of the present paper is that we can avoid the extra assumption of Riemannian approximation which is needed in \cite{CCG} (hypotheses (1.6) and (1.7) in \cite{CCG}), and in fact we prove that our structure hypotheses \eqref{structure-zero} imply that such approximation always exists. The second ingredient is a Poincar\'e-type inequality for smooth functions, which was originally established in \cite{CLM}. This latter estimate, in Lemma \ref{SobXU}, is the only point in the paper where we are forced to impose the limited range $2\le p \le 4$. In fact, we believe that such constraint is not needed for the Caccioppoli inequality in Proposition \ref{cor1} and we plan to return to this point in future work.

To the best of our knowledge, the present paper and \cite{CCG} are the first instances in the literature of the study of higher regularity for weak solutions of the non stationary $p$-Laplacian PDE   in the sub-Riemannian setting. Both are based on  techniques introduced by Zhong in \cite{Zhong}. By contrast, the stationary case has been far more developed, mostly  in the Heisenberg group case. We mention here the contributions of  Domokos \cite{Dom}, Manfredi, Mingione \cite{Manfredi-Mingione}, Mingione, Zatorska-Goldstein and Zhong \cite{Mingione-Zatorska-Zhong}, Ricciotti \cite{Ricciotti}, \cite{Ricciotti1} and eventually those in  \cite{Zhong,MZ}, where   the horizontal $C^{1,\alpha}$ regularity in the full range $1<p<\infty$ is proved. Regularity in more general  contact sub-Riemannian manifolds, including the rototraslation group,  has been recently established by  two of the authors and coauthors \cite{CCLDO} and independently by Mukherjee 
\cite{M} based on an extension of the techniques in \cite{Zhong}. More recently Domokos and Manfredi \cite{DomMan} have studied regularity in higher steps groups and in some special non-group structures.

The structure of the paper is as follows: In Section \ref{preliminaries} we review some preliminary definitions and results from \cite{CCG} and lay out the approximation scheme, thus reducing the problem to finding  estimates for smooth solutions $u_\delta$ of approximating regularized equations, which are stable as $\delta\to 0$. From that point on, we will simplify the notation by dropping the script $\delta$ and by focusing on the case $A_i(x,\xi)=A_i(x)$, thus highlighting how in this case we can obtain more explicit constants in the right hand side of our estimates. In Section \ref{old energy estimates} we recall some energy type estimates from \cite{CCG}.  In Section \ref{new energy estimates} we show that   derivatives of weak solutions along the center are in every $L^q_{loc}$ space, $q\ge 2$, uniformly in $\delta>0$ and establish the key Caccioppoli type inequality in Proposition \ref{cor1}. In this section we need to use the limitation $2\le p \le 4$ in the proof of a Sobolev type estimate. We conjecture that with the exception of  Lemma \ref{SobXU}, all other estimates continue to hold in the range $2\le p<\infty$. Using  Proposition \ref{cor1}, in Section \ref{lipschitz estimate} we prove that the solutions are locally Lipschitz continuous in the subRiemannian metric (i.e. the horizontal gradient is in $L^\infty_{loc}$) uniformly in $\delta$. 
We note explicitly that the Moser iteration in Section
\ref{lipschitz estimate} involves a Sobolev type estimate, which is also stable as $\delta\to 0$, in view of the results in \cite{CC}. Section \ref{time-derivative} addresses the higher integrability of the time derivatives $\p_t u$ of weak solutions.

\noindent {\it Acknowledgements} The authors are grateful to Nicola Garofalo for  conversations around the topics of this paper. Indeed, the results presented here are a development of  some initial work that the authors did jointly with him.

\section{Approximating weak solutions via regularizations}\label{preliminaries}

As mentioned in the introduction, our strategy for the proof of the Lipschitz regularity is to locally 
approximate the weak solutions of \eqref{maineq-zero} with smooth solutions  $u_\delta$ of less degenerate PDE such as \eqref{maineq} and prove  estimates on such approximate solutions that are uniform as $\delta\to 0$.  This approximation is built using both the regularity results in \cite{CCG},  recalled below in Theorem \ref{CCG}, and a Riemannian approximation scheme (see \cite{Montgomery,CC} and references therein). We start by recalling the main points of the latter. 
First, we will use interchangeably the notation $Z$ and $X_{2n+1}$ for the generator of the center of the Lie algebra.  The left invariant sub-Riemannian metric $(\Hn,g_0)$  defined by $\langle X_i, X_j \rangle_0 =\delta_{ij}$, for $i,j=1,...,2n$, can be approximated in the Gromov-Hausdorff sense through a sequence of Riemannian metrics $g_\e$, for $\e\to 0^+$, defined by imposing that 
$X_1,...,X_{2n},\e Z$ is an orthonormal $g_\e-$frame for all $\e>0$. In the terminology of \cite{Montgomery}, the metrics $g_\e$ {\it tame} the metric $g_0$.  We relabel the vectors in this frame as $X_1^\e, ..., X_{2n+1}^\e$. The corresponding gradient
\[\nabla_\e u= \sum_{i=1}^{2n} X_i u X_i + \e^2 Zu Z= \sum_{i=1}^{2n+1} X_i^\e u X_i^\e
\] 
has the obvious property that 
$\nabla_\e u \to (\nabla_0 u, 0)$ as $\e \to 0$.  We note explicitly that 
\[| \nabla_\e u|_{\e}^2:=| \nabla_\e u |_{g_\e}^2=\sum_{i=1}^{2n} (X_i u)^2+ \e^2 (Zu)^2\to |\nabla_0 u|_0,
\] 
as $\e\to 0$. For $\delta>0$, the $\delta-$regularized Riemannian $p-$Laplacian, i.e. the operator related to the Euler-Lagrange equations for the $p$-energy $\int |\nabla_\e u|_\varepsilon^p dx$, is
\begin{equation}\label{epsilon-plap} L_p^\e u := \sum_{i=1}^{2n+1} X_i^\e (  [\delta+| \nabla_\e u|_\e^2]^{\frac{p-2} {2}}X_i^\e u),\end{equation}
and provides a  natural (quasilinear) elliptic regularization of the subelliptic $p$-Laplacian.

Next, we recall the regularity theorem proved in \cite{CCG}.
\begin{thrm}[\cite{CCG}]\label{CCG}  For $\Omega\subset \Hn$, $2\le p < \infty$, and  $\delta> 0$, assume that the
the functions $A_{i,\delta}:\Om\times \R^{2n}\to \R$, $i=1,...,2n$ satisfy the following structure conditions: 

\noindent {\bf (i)} For some $\lambda, \Lambda>0$ depending only on $\lambda', \Lambda'$, one has
\begin{equation}\label{structure1}
\begin{cases}
\lambda (\ddelta+|\xi|^2)^{\frac{p-2}{2}} |\eta|^2 \le  \p_{\xi_j} A_{i, \delta} (x,\xi) \eta_i \eta_j \le \Lambda (\ddelta+|\xi|^2)^{\frac{p-2}{2}} |\eta|^2, 
\\
|A_i (x,\xi)| +  |\p_{x_j} A_{i, \delta} (x,\xi)| \le  \Lambda (\ddelta+|\xi|^2)^{\frac{p-1}{2}}.
\end{cases}
\end{equation}
\noindent {\bf (ii)} We assume that one can approximate $A_{i,\delta}$ by  a $1$-parameter family of regularized approximants $A^\e_\delta(x,\xi)=(A_{1,\delta}^\e(x,\xi),...,A_{2n+1,\delta}^\e(x,\xi))$  defined for a.e. $x\in \Om$ and every $\xi \in \R^{2n+1}$, and such that 
 for a.e. $x\in\Omega, $  for all $\xi =  \sum_{i=1}^{2n} \xi_i X_i^\e + \xi_{2n+1} X_{2n+1}^\e$ , and $\xi^\e =  \sum_{i=1}^{2n} \xi_i X_i^\e + \e \xi_{2n+1} X_{2n+1}^\e$   one has uniformly on compact subsets of $\Om$,
\begin{equation}\label{structure2}
(A_{1,\delta}^\e(x,\xi^\e),...,A_{2n+1,,\delta}^\e(x,\xi^\e))\ \underset{\e\to 0^+}{\longrightarrow}\    (A_{1,,\delta}(x, \xi_1,...,\x_{2n}), ..., A_{2n,\delta} (x, \xi_1,...,\x_{2n}),0),
\end{equation}
and furthermore
\begin{equation}\label{structure-epsilon}
\begin{cases}
\lambda (\ddelta+|\xi|^2)^{\frac{p-2}{2}} |\eta|^2 \le  \p_{\xi_j} A_{i,\delta}^\e(x,\xi) \eta_i \eta_j \le \Lambda (\ddelta+|\xi|^2)^{\frac{p-2}{2}} |\eta|^2, 
\\
|A_{i,\delta}^\e(x,\xi)| +  |\p_{x_j} A_{i,\delta}^\e(x,\xi)| \le  \Lambda (\ddelta+|\xi|^2)^{\frac{p-1}{2}},
\end{cases}
\end{equation}
 for all $\eta\in \R^{2n+1}$, and 
 for some $0<\lambda\le \Lambda <\infty$ independent of $\e$.
 Let $u_\delta\in L^p((0,T), W_{\loc}^{1,p}(\Om))$  be a weak solution of  \begin{equation}\label{maineq}
\p_t u_\delta = \sum_{i=1}^{2n} X_i A_{i, \delta}(x, \nabla_0 u_\delta). \end{equation} in $Q=\Om\times (0,T)$.
 If $\delta>0$ then $u_\delta$  is $C^\infty$ smooth in $Q$.
 \end{thrm}

One of our main contributions in the present paper is that one can avoid the assumptions \eqref{structure2}, and \eqref{structure-epsilon}, and build the Riemannian approximation using solely the structure condition \eqref{structure1}. Our result is stated in the following proposition

\begin{prop}\label{approximation theorem}
Let $A_i$ be as in \eqref{structure-zero}. For every $\delta>0$ there exists $A_{i,\delta}$ such that
\begin{equation}\label{approx-structure}
A_{\delta} (x,\xi)\ \underset{\delta\to 0^+}{\longrightarrow}\    A(x, \xi),
\end{equation} 
satisfying the hypothesis \eqref{structure1}, \eqref{structure2}, and \eqref{structure-epsilon} with constants depending only on the original $\lambda',\Lambda'$.  Moreover, if a function $A_{i,\delta}$ satisfies \eqref{structure1}, then it also satisfies \eqref{structure2}, and \eqref{structure-epsilon} with constants depending only on the original $\lambda^\prime,\Lambda^\prime$.
\end{prop}

In view of Theorem \cite{CCG}, the latter yields immediately the following
\begin{cor}\label{approx}
Let $u$ be a weak solution of \eqref{maineq-zero} in $Q=\Omega\times (0,T)$, with the structure conditions \eqref{structure-zero}. For any sub-cylinder $Q_1=\Omega_1\times (t_1,t_2) \subset \subset  \Omega\times (0,T)$,  there exists a  sequence $\{u_{\delta}\}$  of smooth solutions of the
 regularized problem 
\begin{equation}\label{maineq}
\p_t u_\delta = \sum_{i=1}^{2n} X_i A_{i, \delta}(x, \nabla_0 u_\delta) \quad \quad  \text{ in  \quad }Q_1, \text{ and } u_\delta=u \quad \text{ on }\partial_p  \ Q_1\end{equation}
converging to $u$, as $\delta\to 0^+$,  uniformly on compacts subsets of $Q_1$ and weakly in the $W^{1,p}$-norm.  Here we have denoted by $\partial_p Q_1=\Omega_1\times \{t=t_1\}\cup \partial \Omega_1 \times (t_1,t_2)$ the parabolic boundary of $Q_1$. The functions $A_\delta$ satisfy \eqref{approx-structure}, \eqref{structure1}, \eqref{structure2}, and \eqref{structure-epsilon} with constants depending only on the original $\lambda',\Lambda'$.
\end{cor}
\begin{proof}
Let $u$ be a weak solution of  \eqref{maineq-zero} in $Q$. In view of the results in \cite{ACCN} we know that the 
solution is H\"older continuous in compact subsets.  For $\delta>0$,  let $A_{i,\delta}$ be as in the statement of Proposition \ref{approximation theorem}, and consider  the unique weak  solution 
$u_{\delta}$ of \eqref{maineq}.  In view of the comparison principle, the uniform continuity  and Caccioppoli inequalities for $\{u_\delta\}$ proved in \cite{ACCN} and of \eqref{approx-structure}, \eqref{structure1} one can easily see that $u_\delta \to u$ uniformly on compact subsets of $Q_1$ and weakly 
in the $W^{1,p}$-norm. In order to conclude the proof, we  only need to observe that $u_\delta$ are smooth (with the regularity possibly depending on $\delta>0$ of course) thanks to  
Theorem \ref{approximation theorem} and Theorem \ref{CCG}.
In fact,  one can apply the results from \cite{CCG}, to derive regularity estimates that are uniform in the parameter $\e$, thus  yielding that the family $u_{\delta, \varepsilon}$ has a subsequence converging to a solution of the problem 
\eqref{maineq} as $\varepsilon \to 0$, which coincides with $u_\delta$ in view the comparison principle. 
\end{proof}

We are left with the task of proving Proposition \ref{approximation theorem}. To better illustrate the argument of the proof  we present the special, simpler,  case of the $p$-Laplacian \eqref{plap-para}, i.e. for $\xi  = \sum_{i=1}^{2n} \xi_i X_i\in \R^{2n}$, and $x\in \Hn$,
$$A_i(x,\xi)=|\xi|^{p-2}\xi_i, \quad \quad i=1,...,2n.$$
 In this case, we consider for each $\delta>0$, the  following functions
 \begin{equation}\label{plap-para-delta}
 A_{i,\delta}(x,\xi)= (\delta+|\xi|^2)^{\frac{p-2}{2}}\xi_i, \quad \quad i=1,...,2n,
 \end{equation}
 and for each $\e>0$, and $\xi  = \sum_{i=1}^{2n+1} \xi_i X_i^\e\in \R^{2n+1}$,
 \begin{equation}\label{plap-para-delta-epsilon }
  A_{i,\delta}^\e(x,\xi)= (\delta+\|\xi\|_{g_\e(x)}^2)^{\frac{p-2}{2}}\xi_i,\quad \quad i=1,...,2n+1.
 \end{equation}
 While the quantity $||\cdot||_{g_\e(x)}$ a priori depends on $x\in \Hn$, we remark that when $\xi$ is a left invariant vector field, since $g_\e$ is left invariant as well, the dependence of $||\xi||_{g_\e(x)}$ on the point $x$ vanishes.

\begin{proof}[Proof of  Proposition \ref{approximation theorem}]
Following the intuition from the example above, we construct the approximates through a two steps process.
For $0<\delta<1$, let us define 
 \begin{equation}\label{approx1}
 A_{\delta} (x,\xi)= A (x,\xi) + \lambda\delta ^{\frac{p-2}{2}} \xi 
 \end{equation}
It is clear that 
\begin{equation}\label{approx-structure}
A_{\delta} (x,\xi)\ \underset{\delta\to 0^+}{\longrightarrow}\    A(x, \xi)
\end{equation}
and furthermore, for some $\lambda, \Lambda>0$ depending only on $\lambda', \Lambda'$, one has the estimate \eqref{structure1}.

For each  $\xi  = \sum_{i=1}^{2n+1} \xi_i X_i^\e\in \R^{2n+1}$, and $\e,\delta>0$ we set 
\begin{equation}
\label{approx2} A_{i,\delta, \varepsilon} (x,\xi) =  \tilde A_{i}(x,\xi_H) + \lambda  (\delta + |\xi|_\e^2)^{\frac{p-2}{2}} \xi_i,
\end{equation}
for $i=1,...,2N+1$.
 Here we have denoted $\xi_H=(\xi_1,...,\xi_{2n})$,  $\tilde A=(A,0)\in \R^{2n+1}$, and  
$|\xi|_\e^2= \sum_{i=1}^{2n+1} \xi_i^2 $.

Clearly  
 for a.e. $x\in\Omega, $ and for all $\xi^\e =  \sum_{i=1}^{2n} \xi_i X_i^\e + \e \xi_{2n+1} X_{2n+1}^\e$  one has uniformly on compact subsets of $\Om \times (0,T)$,
$$
(A_{1, \delta, \varepsilon} (x,\xi^\e),...,A_{2n+1, \delta, \varepsilon} (x,\xi^\e))\ \underset{\varepsilon\to 0^+}{\longrightarrow}\    (A_{1, \delta}(x, \xi_1,...,\x_{2n}), ..., A_{2n, \delta} (x, \xi_1,...,\x_{2n}),0),
$$
where $A_\delta$ is defined as in \eqref{approx1}.
In addition one can see that there exist constants $\lambda, \Lambda>0$ depending on $\lambda^\prime, \Lambda^\prime$ such that
$$
\begin{cases}
\lambda (\ddelta+|\xi|_\e^2)^{\frac{p-2}{2}} |\eta|_\e^2 \le  \p_{\xi_j} A_{i, \delta, \e} (x,\xi) \eta_i \eta_j \le \Lambda (\ddelta+|\xi|_\e^2)^{\frac{p-2}{2}} |\eta|_\e^2, 
\\
|A_i (x,\xi)| +  |\p_{x_j} A_{i, \delta,\e} (x,\xi)| \le  \Lambda (\ddelta+|\xi|_\e^2)^{\frac{p-1}{2}},
\end{cases}
$$
 for all $\eta =\sum_{i=1}^{2n+1} \eta_i X_i^\e\in \R^{2n+1}$.
\end{proof}

\begin{rmrk}\label{notation}
For the rest of the paper we will always consider solutions $u_\delta$ of the Dirichlet problem \eqref{maineq} with $\delta>0$,  in a cylinder $ D\times (\tau_1,\tau_2)\subset \subset Q_1$, with $D\subset \subset \Omega_1$ and $[\tau_1,\tau_2]\subset (t_1,t_2)$. For the sake of notation we will drop the subscript $\delta$ from $u_\delta$ and $A_{i,\delta}$, and with a slight abuse of notation write  $Q=\Omega \times (0,T)$ instead of $D\times (\tau_1,\tau_2)$. To further simplify the formulation of the estimates, we will assume that $A_i(x,\xi)= A_i(\xi)$, as in this case we can obtain sharper constants, and so we highlight these more involved aspects of the proofs. The more general case is handled in a similar fashion, and does not lead to explicit constants on the right hand side of the estimates.  We note explicitly that all constants are independent of the parameter $\delta>0$.
\end{rmrk}

\section{Preliminary energy estimates}\label{old energy estimates}
We recall the basic Caccioppoli inequalities proved in \cite{CCG}. These inequalities apply to a smooth solution $u$ of the approximating equation \eqref{maineq} with $\delta>0$,  in a cylinder $ Q\subset \subset Q_1$. In what follows we will implicitly assume that all constants on the right hand side of the inequalities depend on $n, p$, on the structure constants, $\lambda, \Lambda$ {\it but not on $\delta$.} 

\begin{lemma} \label{diff-PDE} Let $u$ be a solution of \eqref{maineq} in $Q$, with $\delta>0$. If we set $v _l = X _lu$, with $ l = 1, 2, . . . , 2n$, and $s_l = (-1)^{[l/n]}$ then the function $v _l$ is  a  solution of
\begin{equation}\label{eqderivX}
\p_t v _l
= \sum_{i,j=1}^{2n} X _i\Big( A_{i, \xi_j}  (\nabla_0u) X _l X _ju \Big) + s_l Z (A_{l+s_ln}  (\nabla_0u)) .\end{equation}

\end{lemma}

\begin{lemma}\label{eqZ} Let $u$ be a solution of \eqref{maineq} in  $Q$, with $\delta>0$. The function
$Zu$  is  then a  solution of the equation
$$\partial_t Zu=
\sum_{i,j=1}^{2n}
X _i(A _{i, \xi_j}(\nabla_0u)X _j Zu).
$$
\end{lemma}

First we recall a Caccioppoli estimates for derivatives of the solution along the center of the group.

\begin{lemma}[Lemma 3.4 \cite{CCG}]\label{stimaZu} Let $u$ be a solution of \eqref{maineq} in $Q$ with $\delta>0$.
For every $ \beta \geq 0$ and  non-negative $\eta\in C^{1}([0,T], C^\infty_0 (\Om))$ vanishing on the parabolic boundary of $Q$, one has
\begin{equation*}
\begin{aligned}
\int_{t_1} ^{t_2}\int_\Omega 
(\delta+|\nabla_0u|^2)^{\frac{p-2}{2}}|Zu|^{\beta} | \nabla_0Zu|^2\eta^{4+\beta}
\le & C
\int_{t_1} ^{t_2}\int_\Omega
(\delta+|\nabla_0u|^2)^{\frac{p-2}{2}}|Zu|^{\beta+2}|\nabla_0\eta|^2 \eta^{2+\beta}\\
&
+ C\int_{t_1}^{t_2}\int_\Omega |Zu|^{\beta+2} |\partial_t \eta |  \eta^{3+\beta},
\end{aligned}
\end{equation*}
where $C=C(\lambda, \Lambda)>0$.
\end{lemma}

Second, we recall a Caccioppoli estimate for the horizontal derivatives. 

\begin{lemma}[Lemma 3.5 \cite{CCG}] \label{lemma3.4} Let $u$ be a weak solution of \eqref{maineq} in $Q$, with $\delta>0$. For every  $\beta \geq 0 $ and non-negative $\eta\in C^{1}([0,T], C^\infty_0 (\Om))$ vanishing on the parabolic boundary of $Q$,
we have
\begin{equation*}
\begin{aligned}
&\frac{1}{\beta+2}\sup_{t_1<t<t_2}\int_\Omega (\delta+ |\nabla_0u|^2)^{\frac{\beta+2}{2}}\eta^2 +\int_{t_1} ^{t_2} \int_\Omega
(\delta + |\nabla_0u|^2)^{(p-2+\beta)/2} |\nabla_0^2u|^2 \eta^2\\
\leq & \, C
\int_{t_1} ^{t_2} \int_\Omega
(\delta + |\nabla_0u|^2)^{\frac{p+\beta}{2}} ( |\nabla_0\eta|^2 + |Z\eta| \eta)
+ 
\frac{C}{\beta+2}
\int_{t_1} ^{t_2} \int_\Omega 
(\delta + |\nabla_0u|^2)^{\frac{\beta+2}{2}}|\partial_t \eta| \eta \\
& +
C(\beta + 1)^4
\int_{t_1} ^{t_2} \int_\Omega
(
\delta + |\nabla_0u|^2)^{\frac{p-2+\beta}{
2}} |Zu|^2 \eta^2,
\end{aligned}
\end{equation*}
where $C = C(n, p, \lambda, \Lambda) > 0$, independent of $\delta$.
\end{lemma}

%
%
%
%
%
%

\section{Main Caccioppoli inequality}\label{new energy estimates}

The main result of this section is a Caccioppoli inequality, Proposition \ref{cor1}, for the horizontal derivatives of the weak solutions of \eqref{maineq}, with $\delta>0$. To do this,  we first need to prove an  estimate for  the derivative along the center $Zu$ in Lemma \ref{SobXU} and Lemma \ref{mainlemma}.  All estimates are  uniform in $\delta>0$, and the constants are stable as $\delta\to 0$.

We begin by recalling a Poincar\'e-like interpolation inequality from \cite{CLM}. In the proof, we will need the restriction $2\le p\le 4$ and this is the only use we make of this hypothesis in the paper.

\begin{lemma}\label{SobXU} Assume that $2\le p \leq 4$ and let $u\in C^2(Q)$.
There exists a constant $C>0$ depending only on $n, p$ such that for every $ \beta \geq 0$ and non-negative $\eta\in C^{1}([0,T], C^\infty_0 (\Om))$ vanishing on the parabolic boundary of $Q$, we have
\begin{equation}\label{interpolation}
\begin{aligned}
\int_{t_1} ^{t_2}\int_\Omega &
|Zu|^{p+\beta} \eta^{p+\beta} 
\le  C(p+\beta) ||\nabla_0\eta||_{L^\infty}\int\int_{spt(\eta)} (\delta+ |\nabla_0u|^2)^{\frac{p+\beta}{2}}\\
&  + 
C(p+\beta)\int_{t_1} ^{t_2}\int_\Omega 
(\delta + |\nabla_0u|^2)^{(p-2)/2}
|Zu|^{\beta} | \nabla_0Zu|^2\eta^{4+\beta}
\end{aligned}
\end{equation}

\end{lemma}
\begin{proof}
We denote 
\begin{equation}\label{LR}
L=\int_{t_1} ^{t_2}\int_\Omega
|Zu|^{p+\beta} \eta^{p+\beta}, \quad R=\int\int_{spt(\eta)} (\delta+ |\nabla_0u|^2)^{\frac{p+\beta}{2}}.
\end{equation}
We estimate $L$ from above as follows. Fix $l=1,2,...,n$. Note that
\[
Zu=X_lX_{n+l}u-X_{n+l}X_lu.
\]
We can write 
\[ |Zu|^{p+\beta}=|Zu|^{p-2+\beta}Zu (X_lX_{n+l}u-X_{n+l}X_lu).
\]
Then integration by parts gives us
\begin{equation}\label{ML1}
\begin{aligned}
L=&\int_{t_1} ^{t_2}\int_\Omega |Zu|^{p-2+\beta}Zu (X_lX_{n+l}u-X_{n+l}X_lu)\eta^{p+\beta}\\
=&-(p-1+\beta)\int_{t_1} ^{t_2}\int_\Omega |Zu|^{p-2+\beta}(X_lZuX_{n+l}u-X_{n+l}ZuX_lu)\eta^{p+\beta}\\
&-(p+\beta)\int_{t_1} ^{t_2}\int_\Omega|Zu|^{p-2+\beta}Zu(X_{n+l}uX_l\eta-X_luX_{n+l}\eta)\eta^{p-1+\beta}\\
\le & \, 2(p+\beta)\int_{t_1} ^{t_2}\int_\Omega |\nabla_0u||Zu|^{p-2+\beta}|\nabla_0Zu|\eta^{p+\beta}\\
&+2(p+\beta)\int_{t_1} ^{t_2}\int_\Omega |\nabla_0u||Zu|^{p-1+\beta}|\nabla_0\eta|\eta^{p-1+\beta}=I_1+I_2
\end{aligned}
\end{equation}
We will estimate the integrals $I_1,I_2$ on the right hand side of \eqref{ML1} by H\"older's inequality. First for $I_2$, we have
\begin{equation}\label{ML2}
I_2\le 2(p+\beta) ||\nabla_0 \eta||_{L^\infty} R^{\frac{1}{p+\beta}}L^{\frac{p-1+\beta}{p+\beta}},
\end{equation}
where $L$ and $R$ are as in \eqref{LR}.

Second, for $I_1$, we have
\begin{equation}\label{ML3}
I_1\le 2(p+\beta) M^{\frac{1}{2}}R^{\frac{4-p}{2(p+\beta)}} L^{\frac{2p-4+\beta}{2(p+\beta)}},
\end{equation}
where 
\[
M=\int_{t_1} ^{t_2}\int_\Omega 
|\nabla_0u|^{p-2}|Zu|^{\beta} | \nabla_0Zu|^2\eta^{4+\beta}.
\]

This yields 
\begin{equation}\label{last-eq}
L\le C(p+\beta) ||\nabla_0 \eta||_{L^\infty} R^{\frac{1}{p+\beta}}L^{\frac{p-1+\beta}{p+\beta}}
+C (p+\beta) M^{\frac{1}{2}}R^{\frac{4-p}{2(p+\beta)}} L^{\frac{2p-4+\beta}{2(p+\beta)}},
\end{equation}
from which the conclusion follows immediately through Young's inequality.
\end{proof}

The previous Poincar\'e-like inequality can be applied to solutions of \eqref{maineq} and through invoking  Lemma \ref{stimaZu} lead us to the following key estimate.

\begin{lemma}\label{mainlemma}
Let $u$ be a solution of \eqref{maineq} in $Q$, with $\delta>0$ and $2\leq p \leq 4$. 
Then for every $ \beta \geq 0$ and non-negative $\eta\in C^{1}([0,T], C^\infty_0 (\Om))$ vanishing on the parabolic boundary of $Q$, we have
\begin{equation}\label{mainestimate}
\begin{aligned}
\Big(\int_{t_1} ^{t_2}\int_\Omega &
|Zu|^{p+\beta} \eta^{p+\beta} \Big)^{\frac{1}{p+\beta}}
\le  C(p+\beta) ||\nabla_0\eta||_{L^\infty}\Big(\int\int_{spt(\eta)} (\delta+ |\nabla_0u|^2)^{\frac{p+\beta}{2}}\Big)^{\frac{1}{p+\beta}}\\
&  + 
C(p+\beta)||\eta\partial_t\eta||_{L^\infty}^{\frac{1}{2}}|spt(\eta)|^{\frac{p-2}{2(p+\beta)}}\Big(\int\int_{spt(\eta)}  (\delta+ |\nabla_0u|^2)^{\frac{p+\beta}{2}} \Big)^{\frac{4-p}{2(p+\beta)}}
\end{aligned}
\end{equation}
\end{lemma}
\begin{rmrk}\label{limit-remark}  Suspending temporarily the notation established in Remark \ref{notation}, we denote by $u_\delta$ the solutions of the approximating equation \eqref{maineq}.  In particular, Lemma \ref{mainlemma} establishes the local $L^q$  integrability of  $Z u_\delta$, the derivative along the center of the approximating solutions, with uniform $L^q$ bounds as $\delta\to 0$, for all $1\le q<\infty$. This implies that  one can find a subsequence, $Z u_{\delta_k}$  converging to a $L^q_{\loc}$ function, which in view of the definition of weak derivative, is also a derivative along the center of the uniform limit of the $u_\delta$. Since such limit is the original solution of \eqref{maineq-zero}, this proves the local  integrability of $Zu$ in every $L^q$ class as stated in Theorem \ref{main1}.
\end{rmrk}

\begin{proof}
We apply the inequality \eqref{last-eq} in the previous lemma to the solution $u$ and invoke Lemma \ref{stimaZu} to estimate the integral
\[
M=\int_{t_1} ^{t_2}\int_\Omega 
|\nabla_0u|^{p-2}|Zu|^{\beta} | \nabla_0Zu|^2\eta^{4+\beta},
\]
  obtaining
\begin{equation}\label{ML4}
\begin{aligned}
M &\le C
\int_{t_1} ^{t_2}\int_\Omega
(\delta+|\nabla_0u|^2)^{\frac{p-2}{2}}|Zu|^{\beta+2}|\nabla_0\eta|^2 \eta^{2+\beta}
+ C\int_{t_1}^{t_2}\int_\Omega |Zu|^{\beta+2} |\partial_t \eta |  \eta^{3+\beta}\\
&\le C ||\nabla_0 \eta||_{L^\infty}^2 R^{\frac{p-2}{p+\beta}}L^{\frac{\beta+2}{p+\beta}}
+C||\eta\partial_t\eta||_{L^\infty}|spt(\eta)|^{\frac{p-2}{p+\beta}}L^{\frac{\beta+2}{p+\beta}},
\end{aligned}
\end{equation}
where $C=C(\lambda, \Lambda)>0$, and $L$ is as in \eqref{ML1}. In the second inequality of \eqref{ML4}, we used H\"older's inequality.
Combining \eqref{ML3}, \eqref{last-eq} and \eqref{ML4}, we obtain the estimate for $I_1$,
\begin{equation}\label{ML5}
I_1\le C(p+\beta) ||\nabla_0 \eta||_{L^\infty} R^{\frac{1}{p+\beta}}L^{\frac{p-1+\beta}{p+\beta}}
+C (p+\beta) ||\eta\partial_t\eta||_{L^\infty}^{\frac 1 2}|spt(\eta)|^{\frac{p-2}{2(p+\beta)}} R^{\frac{4-p}{2(p+\beta)}}
L^{\frac{p-1+\beta}{2(p+\beta)}}.
\end{equation}
Next, we substitute the latter in  the estimate \eqref{ML3} for $I_1$ and the estimate \eqref{ML2} for $I_2$ to \eqref{ML1}, and conclude
\[ 
L\le C(p+\beta) ||\nabla_0 \eta||_{L^\infty} R^{\frac{1}{p+\beta}}L^{\frac{p-1+\beta}{p+\beta}}
+C (p+\beta) ||\eta\partial_t\eta||_{L^\infty}^{\frac 1 2}|spt(\eta)|^{\frac{p-2}{2(p+\beta)}} R^{\frac{4-p}{2(p+\beta)}}
L^{\frac{p-1+\beta}{2(p+\beta)}},
\]
which yields immediately  \eqref{mainestimate}. \end{proof}

The following result follows from  Lemma \ref{mainlemma}, and the energy estimate in Lemma \ref{lemma3.4}. It yields a Caccippoli inequality for the horizontal derivatives of weak solutions, which extends to the subRiemannian setting the analogue  Euclidean estimate   proved in \cite[Proposition 3.2 (3.7), page 225]{DB}.

\begin{prop}\label{cor1}
Let $u$ be a weak solution of \eqref{maineq} in $Q$, with $\delta>0$ and $2\leq p \leq 4$. 
Then for every $ \beta \geq 0$ and non-negative $\eta\in C^{1}([0,T], C^\infty_0 (\Om))$ vanishing on the parabolic boundary of $Q$, we have
\begin{equation}\label{maincacci}
\begin{aligned}
&\sup_{t_1<t<t_2}\int_\Omega (\delta+ |\nabla_0u|^2)^{\frac{\beta+2}{2}}\eta^2 +\int_{t_1} ^{t_2} \int_\Omega
(\delta + |\nabla_0u|^2)^{(p-2+\beta)/2} |\nabla_0^2u|^2 \eta^2\\
\leq & C(p+\beta)^7\big( ||\nabla_0\eta||_{L^\infty}^2+||\eta Z\eta||_{L^\infty}\big)\int\int_{spt(\eta)} (\delta+ |\nabla_0u|^2)^{\frac{p+\beta}{2}}\\
& + 
C(p+\beta)^7||\eta\partial_t\eta||_{L^\infty}|spt(\eta)|^{\frac{p-2}{p+\beta}}\Big(\int\int_{spt(\eta)}  (\delta+ |\nabla_0u|^2)^{\frac{p+\beta}{2}} \Big)^{\frac{\beta+2}{p+\beta}},
\end{aligned}
\end{equation}
where $C=C(n,p,\lambda,\Lambda)>0$.
\end{prop}
\begin{rmrk} Although the statement addresses the approximate solution $u_\delta$, in view of   arguments analogue to those in  Remark \ref{limit-remark}, the same estimate holds for weak solutions of \eqref{maineq-zero}.
\end{rmrk}
\begin{proof}
Lemma \ref{lemma3.4} gives us the following estimate for the left hand side of \eqref{maincacci}
\begin{equation}\label{NH1}
\begin{aligned}
&\sup_{t_1<t<t_2}\int_\Omega (\delta+ |\nabla_0u|^2)^{\frac{\beta+2}{2}}\eta^2 +\int_{t_1} ^{t_2} \int_\Omega
(\delta + |\nabla_0u|^2)^{(p-2+\beta)/2} |\nabla_0^2u|^2 \eta^2\\
\leq & \, C (p+\beta)
\int_{t_1} ^{t_2} \int_\Omega
(\delta + |\nabla_0u|^2)^{\frac{p+\beta}{2}} ( |\nabla_0\eta|^2 + |Z\eta| \eta)
+ 
C
\int_{t_1} ^{t_2} \int_\Omega 
(\delta + |\nabla_0u|^2)^{\frac{\beta+2}{2}}|\partial_t \eta| \eta \\
& +
C(p+\beta )^5
\int_{t_1} ^{t_2} \int_\Omega
(
\delta + |\nabla_0u|^2)^{\frac{p-2+\beta}{
2}} |Zu|^2 \eta^2.
\end{aligned}
\end{equation}

To obtain the desired estimate \eqref{NH1}, we will show that each integral on the right hand side of \eqref{NH1} can be bounded from above by the right hand side of \eqref{maincacci}. For the first integral on the right hand of \eqref{NH1}, it is obviously bounded from above by the first item on the right hand side of \eqref{maincacci}. For the second integral, H\"older's inequality gives us
\[
\int_{t_1} ^{t_2} \int_\Omega 
(\delta + |\nabla_0u|^2)^{\frac{\beta+2}{2}}|\partial_t \eta| \eta 
\le ||\eta\partial_t\eta||_{L^\infty} |spt(\eta)|^{\frac{p-2}{p+\beta}}\Big(\int\int_{spt(\eta)}  (\delta+ |\nabla_0u|^2)^{\frac{p+\beta}{2}} \Big)^{\frac{\beta+2}{p+\beta}},
\]
which shows that it is bounded from above by the second item on the right hand side of \eqref{maincacci}.

For the third integral on the right hand side of \eqref{NH1}, we use H\"older's inequality and our main lemma, Lemma \ref{mainlemma},  and we have
\begin{equation*}
\begin{aligned}
\int_{t_1} ^{t_2} \int_\Omega
(\delta + |\nabla_0u|^2)^{\frac{p-2+\beta}{2}} |Zu|^2 \eta^2 &\le \Big(\int\int_{spt(\eta)}  (\delta+ |\nabla_0u|^2)^{\frac{p+\beta}{2}}\Big)^{\frac{p-2+\beta}{p+\beta}}\Big(\int_{t_1} ^{t_2}\int_\Omega 
|Zu|^{p+\beta} \eta^{p+\beta} \Big)^{\frac{2}{p+\beta}}\\
&\le C(p+\beta)^2 ||\nabla_0\eta||_{L^\infty}^2\int\int_{spt(\eta)} (\delta+ |\nabla_0u|^2)^{\frac{p+\beta}{2}}\\
&\ + 
C(p+\beta)^2||\eta\partial_t\eta||_{L^\infty}|spt(\eta)|^{\frac{p-2}{p+\beta}}\Big(\int\int_{spt(\eta)}  (\delta+ |\nabla_0u|^2)^{\frac{p+\beta}{2}} \Big)^{\frac{\beta+2}{p+\beta}},
\end{aligned}
\end{equation*}
which concludes the proof of the lemma.
\end{proof}

\section{Boundedness of the horizontal gradient}\label{lipschitz estimate}
In this section we conclude the proof of Theorem \ref{main1}, i.e. we establish that weak solutions of the approximating equation \eqref{maineq} with $\delta>0$ are Lipschitz continuous with respect to the subRiemannian distance, uniformly in the parameter $\delta$.  The proof follows immediately from Proposition \ref{cor1} and from the Moser type iteration in Theorem \ref{main2} below. The proof of Theorem \ref{main2} should be known, but we can not find the precise reference in the literature. It is similar to the proof of Theorem 4 in \cite{Ch} for the case $1<p<2$. The proof is included for the reader's convenience.

First, we recall a few definitions needed in the proof. We will denote by  $d_0(x,y)=||y^{-1}x||$ the subRiemannian distance, where 
$$||x||^4=(\sum_{i=1}^{2n} x_i^2)^2+ 16 x_{2n+1}^2,$$
is the Koranyi gauge.
The corresponding parabolic metric is $d_0 ((x,t), (y,s))= d_0(x,y)+|t-s|^2$. 

\begin{dfn}\label{cylinder} A parabolic cylinder $Q_{r}(x_0, t_0)\subset Q$
is a set of the form
$Q_{r}(x_0, t_0)= B(x_0, r) \times (t_0-r^2, t_0 ).$
where $r>0$, $B(x_0, r)=\{y|\ ||yx_0^{-1}||<r\} \subset \Omega$ denotes the gauge ball of center $x_0$. 
The parabolic boundary of the cylinder $Q_{ r}(x_0, t_0)\subset Q$ is
the  set $B(x_0, r) \times \{t_0-r^2\} \cup \partial B(x_0, r) \times [t_0-r^2, t_0).$
For $r$, $\mu>0$ we also define the cylinders 
$$ Q_{\mu, r} : = B(x, R) \times [t_0 - \mu R, t_0] $$
\end{dfn}

\begin{thrm}\label{main2}  
Let $u\in C^\infty(Q)$, with $Q=\Om\times (0,T)$. If $u$ satisfies the Caccioppoli type inequality \eqref{maincacci}, then  for every $p\ge 2$, and  for any $Q_{\mu, 2r}\subset Q$, we have 
\begin{equation}\label{desired}
\sup_{Q_{\mu, r}}  |\nabla_0u|
\le C \mu^{\frac 1 2} \max\Big( \Big(\frac{1}{\mu r^{N+2}}\int\int_{Q_{\mu,2r}} (\delta+|\nabla_0u|^2)^{\frac{p}{2}}\Big)^{\frac{1}{2}}, \mu^{\frac{p}{2(2-p)}}\Big),
\end{equation}
where $C=C(n,p,\lambda, \Lambda)>0$. 
\end{thrm}
\begin{rmrk} Suspending temporarily the notation established in Remark \ref{notation}, we denote by $u_\delta$ the solutions of the approximating equation \eqref{maineq}.  As mentioned earlier, there is a subsequence $u_\delta\to u$ converging uniformly in compact subsets of $Q$ to the  weak solution $u$ of \eqref{maineq-zero}. In view of the uniform  bound on the Lipschitz constant of $u_\delta$ in \eqref{desired}, then the Lipschitz regularity of $u$ follows immediately.
\end{rmrk}
\begin{proof}
Let  $\eta\in C^{1}([0,T], C^\infty_0 (\Om))$ be a non-negative cut-off function vanishing on the parabolic boundary of $Q$ such that
$|\eta|\le 1$ in $Q$. For $\beta\ge 0$, we set
\[
v=(\delta+|\nabla_0u|^2)^{\frac{p+\beta}{4}}\eta^2.
\]
Then the Caccioppoli inequality \eqref{maincacci} gives us
\begin{equation}\label{maincacci2}
\begin{aligned}
\sup_{t_1<t<t_2}\int_\Omega v^m +\int_{t_1} ^{t_2} \int_\Omega 
|\nabla v|^2
\leq &  C(p+\beta)^7\big( ||\nabla_0\eta||_{L^\infty}^2+||\eta Z\eta||_{L^\infty}\big)\int\int_{spt(\eta)} v^2\\
&+
C(p+\beta)^7||\eta\partial_t\eta||_{L^\infty}|spt(\eta)|^{\frac{p-2}{p+\beta}}\Big(\int\int_{spt(\eta)} v^2\Big)^{\frac{\beta+2}{p+\beta}},
\end{aligned}
\end{equation}
where $C=C(n,p,\lambda,\Lambda)>0$. Here $m=2(\beta+2)/(p+\beta)$. Note that $4/p<m \le 2$. Now let $q=2(m+N)/N$. We have
\[
\int_{t_1} ^{t_2} \int_\Omega v^q\le  \int_{t_1}^{t_2}\Big(
\int_\Omega v^m\Big)^{\frac 2 N}\Big(\int_\Omega v^{\frac{2N}{N-2}}\Big)^{\frac{N-2}{N}}\le C \Big(\sup_{t_1<t<t_2}\int_\Omega v^m\Big)^{\frac 2 N} \Big(\int_{t_1} ^{t_2} \int_\Omega 
|\nabla v|^2\Big),
\]
where $C=C(n)>0$. Here in the second inequality, we used the Sobolev inequality in the space variables. Now we plug the estimate \eqref{maincacci2} into the above inequality and we obtain that
\begin{equation}\label{MM1}
\begin{aligned}
\Big(\int_{t_1} ^{t_2} \int_\Omega v^q\Big)^{\frac{N}{N+2}}\le &  C(p+\beta)^7\big( ||\nabla_0\eta||_{L^\infty}^2+||\eta Z\eta||_{L^\infty}\big)\int\int_{spt(\eta)} v^2\\
&+
C(p+\beta)^7||\eta\partial_t\eta||_{L^\infty}|spt(\eta)|^{\frac{p-2}{p+\beta}}\Big(\int\int_{spt(\eta)} v^2\Big)^{\frac{\beta+2}{p+\beta}},
\end{aligned}
\end{equation}
where $C=C(n,p,\lambda,\Lambda)>0$.  Here $q=2+4(\beta+2)/(N(p+\beta))$. This is the inequality on which our iteration is based.

Let $Q_{\mu, 2r}\subset Q$. We define, for $i=0,1,2,...$, a sequence of radius
$r_i=(1+2^{-i})r$ and a  sequence of exponent $\beta_i$ such  that $\beta_0=0$ and
\[ 
p+\beta_{i+1}=(p+\beta_i) \big(1+\frac{2(\beta_i+2)}{N(p+\beta_i)}\big),
\]
that is,
\[
\beta_i=2(\kappa^i-1), \quad \kappa=\frac{N+2}{N}.
\]
We denote $Q_i=Q_{\mu, r_i}$. Note that $Q_0=Q_{\mu, 2r}$ and $Q_\infty=Q_{\mu, r}$. The we choose a standard parabolic cut-off function $\eta_i\in C^\infty(Q_i)$ such that $\eta_i=1$ in $Q_{i+1}$ with 
\[ 
|\nabla_0 \eta_i|\le \frac{2^{i+8}}{r}, \quad |Z\eta_i|\le  \frac{2^{2i+8}}{r^2}, \quad |\partial_t \eta_i|\le \frac{2^{2i+8}}{\mu r^2} \quad \text{in }Q_i.
\]
Now we let $\eta=\eta_i$ and $\beta=\beta_i$ in \eqref{MM1} and we obtain that for $i=0,1,...$
\begin{equation}\label{MM2}
\begin{aligned}
\Big(\int\int_{Q_{i+1}} (\delta+|\nabla_0u|^2)^{\frac{\alpha_{i+1}}{2}}\Big)^{\frac{N}{N+2}}\le & C2^{2i}\alpha_i^7r^{-2}
\Big[ \Big(\int\int_{Q_i}(\delta+|\nabla_0u|^2)^{\frac{\alpha_{i}}{2}}\Big)^{\frac{p-2}{\alpha_i}}+\\
&+
\mu^{-1}\big(\mu r^{N+2}\big)^{\frac{p-2}{\alpha_i}}
\Big] \Big( \int\int_{Q_i}  (\delta+|\nabla_0u|^2)^{\frac{\alpha_{i}}{2}}\Big)^{\frac{\alpha_i-p+2}{\alpha_i}},
\end{aligned}
\end{equation}
where $C=C(n,p,\lambda,\Lambda)>0$ and $\alpha_i=p+\beta_i=p-2+2\kappa^i$. We denote
\[
M_i=\Big(\frac{1}{\mu r^{N+2}}\int\int_{Q_i} (\delta+|\nabla_0u|^2)^{\frac{\alpha_{i}}{2}}\Big)^{\frac{1}{\alpha_i}}.
\]
Then  we can write \eqref{MM2} as
\[
M_{i+1}^{\frac{\alpha_{i+1}}{\kappa}}\le C \mu^{\frac{2}{N+2}}2^{2i}\alpha_i^7\big(M_i^{p-2}+\mu^{-1}\big)M_i^{\alpha_i-p+2}.
\]
We set 
\[
{\overline M}_i=\max\big( M_i, \mu^{\frac{1}{2-p}}\big).
\]
Then it follows from the above inequality that
\begin{equation}\label{iter}
{\overline M}_{i+1}^{\frac{\alpha_{i+1}}{\kappa}}\le C \mu^{\frac{2}{N+2}}2^{2i}\alpha_i^7\overline{M}_i^{\alpha_i},
\end{equation}
since we may assume that $C=C(n,p, \lambda,\Lambda)\ge 1$. 
Iterating (\ref{iter}), we obtain that
\[
{\overline M}_{i+1}\le \Big(\prod_{j=0}^i K_j^{\frac{\kappa^{i+1-j}}{\alpha_{i+1}}}\Big){\overline M}_0^{\frac{\alpha_0\kappa^{i+1}}{\alpha_{i+1}}},
\]
where 
\[
K_i=C \mu^{\frac{2}{N+2}}2^{2i}\alpha_i^7.
\]
Recall that $\alpha_i=p-2+2\kappa^i$ and $\kappa= (N+2)/N$. Let $i$ go to infinity. We obtain that
\begin{equation}\label{MM3}
{\overline M}_\infty=\limsup_{i\to \infty} {\overline M}_i\le C \mu^{\frac{1}{2}}{\overline M}_0^{\frac p 2},
\end{equation}
where $C=(n,p,\lambda,\Lambda)>0$. Note that
\[
{\overline M}_\infty \ge \sup_{Q_{\mu,r}}|\nabla_0 u|, \quad {\overline M}_0=\max \Big( \Big(\frac{1}{\mu r^{N+2}}\int\int_{Q_{\mu,2r}} (\delta+|\nabla_0u|^2)^{\frac{p}{2}}\Big)^{\frac{1}{p}}, \mu^{\frac{1}{2-p}}\Big).
\]
Thus \eqref{MM3} gives us the desired inequality \eqref{desired}, completing the proof.
\end{proof}

\section{Higher integrability of $\partial_t u$}\label{time-derivative}

In this section, we prove that the time derivative $\partial_tu$ of weak solutions of \eqref{maineq} in the range $2\le p\le 4$ belongs to 
$L^q_{\loc}(\Omega\times (0,T))$ for every $q\ge 1$.  As observed before, once we  establish uniform estimates for $\p_t u_\delta$,  then arguing as in Remark \ref{limit-remark},  one can readily  conclude the integrability of $\p_t u$.

\begin{lemma} Let $u$ be a solution of equation \eqref{maineq} in $Q=\Omega\times (0,T)$. Then we have
\[
\partial_t u\in L^q_{\loc}(\Omega\times (0,T))
\]
for every $q\ge 1$. Moreover, for every $\beta\ge 0$, and non-negative $\eta\in C^1([0,T], C^\infty_0(\Omega))$ vanishing on the 
parabolic boundary, we have
\begin{equation}\label{tderivative}
\int_{t_1}^{t_2}\int_\Omega|\partial_tu|^{\beta+2}\eta^{\beta+2}\le  C\big(M^{2p-2}||\nabla_0\eta||_{L^\infty}^2+M^p || \eta\partial_t\eta||_{L^\infty}\big)^{\frac{\beta+2}{2}}\vert spt(\eta)\vert,
\end{equation}
where $C=C(p,\lambda, \Lambda, \beta)>0$ and $M=\sup_{spt(\eta)}(\delta+|\nabla_0u|^2)^{\frac{1}{2}}$.
\end{lemma}

\begin{proof}
Let $\beta\ge 0$. Since $u$ is a solution of \eqref{maineq}, we can write
\[
|\partial_tu|^{\beta+2}=|\partial_tu|^\beta \partial_tuX_i(A_{i,\delta}(x,\nabla_0u)).
\]
We denote by $L$ the integral on the left hand side of \eqref{tderivative}, which is the object we will estimate. Let $\eta\in C^1([0,T], C^\infty_0(\Omega))$ 
be a non-negative cut-off function, vanishing on the 
parabolic boundary. Since $\eta(\cdot, t)\in C^\infty_0(\Omega)$ for every $t\in [0,T]$,  integration by parts gives us that
\begin{equation}\label{test1}
\begin{aligned}
L=&\int_{t_1}^{t_2}\int_\Omega|\partial_tu|^{\beta+2}\eta^{\beta+2}=\int_{t_1}^{t_2}\int_\Omega|\partial_tu|^{\beta}\partial_tuX_i(A_{i,\delta}(x,\nabla_0u))\eta^{\beta+2}\\
=&-(\beta+2)\int_{t_1}^{t_2}\int_\Omega|\partial_tu|^\beta\partial_tuA_{i,\delta}(x,\nabla_0u)\eta^{\beta+1}X_i\eta\\
&-(\beta+1)\int_{t_1}^{t_2}\int_\Omega|\partial_tu|^\beta X_i(\partial_tu)A_{i,\delta}(x,\nabla_0u)\eta^{\beta+2}=I_1+I_2.
\end{aligned}
\end{equation}
We will estimate the integrals $I_1,I_2$ in the right hand side of the above equality as follows. First, we use the structure condition and H\"older's inequality to estimate $I_1$. We have
\begin{equation}\label{test2}
\begin{aligned}
|I_1|\le & (\beta+2)\Lambda\int_{t_1}^{t_2}\int_\Omega (\delta+|\nabla_0u|^2)^{\frac{p-1}{2}}|\partial_tu|^{\beta+1}\eta^{\beta+1}|\nabla_0\eta|\\
\le & (\beta+2)\Lambda \Big( \int_{t_1}^{t_2}\int_\Omega |\partial_tu|^{\beta+2}\eta^{\beta+2}\Big)^{\frac{\beta+1}{\beta+2}}\times\\
& \times\Big(\int_{t_1}^{t_2}\int_\Omega
(\delta+|\nabla_0u|^2)^{\frac{p-1}{2}(\beta+2)}|\nabla_0\eta|^{\beta+2}\Big)^{\frac{1}{\beta+2}}\\
=&(\beta+2)\Lambda
||\nabla_0\eta||_{L^\infty} |spt(\eta)|^{\frac{1}{\beta+2}} M^{p-1} L^{\frac{\beta+1}{\beta+2}},
\end{aligned}
\end{equation}
where $M=\sup_{spt(\eta)}(\delta+|\nabla_0u|^2)^{\frac{1}{2}}$.

Second, we also use the structure condition and H\"older's inequality to estimate $I_2$.
We have
\begin{equation}\label{test3}
\begin{aligned}
|I_2|\le & (\beta+1)\Lambda\int_{t_1}^{t_2}\int_\Omega  (\delta+|\nabla_0u|^2)^{\frac{p-1}{2}}|\partial_tu|^{\beta}|\nabla_0\partial_tu|\eta^{\beta+2}\\
\le & (\beta+1)\Lambda  
\Big(\int_{t_1}^{t_2}\int_\Omega |\partial_tu|^{\beta+2}\eta^{\beta+2}\Big)^{\frac{\beta}{2(\beta+2)}}\Big(\int\int_{spt(\eta)}(\delta+|\nabla_0u|^2)^{\frac{p}{4}(\beta+2)}\Big)^{\frac{1}{\beta+2}} J^{\frac 1 2}\\
\le & (\beta+1)\Lambda  |spt(\eta)|^{\frac{1}{\beta+2}}
M^{\frac{p}{2}}L^{\frac{\beta}{2(\beta+2)}}J^{\frac 1 2},
\end{aligned}
\end{equation}
where 
\[
J=\int_{t_1}^{t_2}\int_\Omega (\delta+|\nabla_0u|^2)^{\frac{p-2}{2}}|\partial_tu|^\beta|\nabla_0\partial_tu|^2\eta^{\beta+4}.
\]
To estimate the integral $J$, we differentiate equation \eqref{maineq} with respect to $t$ and we obtain that
\[
\partial_t(\partial_tu)=X_i(\partial_{\xi_j}A_{i,\delta}(x,\nabla_0u)X_j(\partial_tu)).
\]
Then we use $\varphi=|\partial_tu|^\beta\partial_tu\eta^{\beta+4}$ as a test function to the above equation and we obtain the following
Caccioppoli inequality by the structure condition and the Cauchy-Schwarz inequality
\begin{equation}\label{test4}
\begin{aligned}
J\le &C \int_{t_1}^{t_2}\int_\Omega (\delta+|\nabla_0u|^2)^{\frac{p-2}{2}}|\partial_tu|^{\beta+2}\eta^{\beta+2}|\nabla_0\eta|^2
+C \int_{t_1}^{t_2}\int_\Omega|\partial_tu|^{\beta+2}\eta^{\beta+3}|\partial_t\eta|\\
\le & C (M^{p-2} ||\nabla_0\eta||_{L^\infty}^2+||\eta\partial_t\eta||_{L^\infty})L,
\end{aligned}
\end{equation}
where $C=C(p,\lambda,\Lambda,\beta)>0$. Here we used the fact that $\eta$ vanishes on the parabolic boundary. 
Combining \eqref{test3} and \eqref{test4}, we obtain the following estimate for $I_2$.
\begin{equation}\label{test5}
|I_2|\le C  |spt(\eta)|^{\frac{1}{\beta+2}}
M^{\frac{p}{2}}L^{\frac{\beta+1}{\beta+2}}(M^{p-2} ||\nabla_0\eta||_{L^\infty}^2+||\eta\partial_t\eta||_{L^\infty})^{\frac 1 2}.
\end{equation}
Now we combine \eqref{test1} with the estimates \eqref{test2} and \eqref{test5} and we end up with
\[
L\le C M^{p-1}||\nabla_0\eta||_{L^\infty} |spt(\eta)|^{\frac{1}{\beta+2}} L^{\frac{\beta+1}{\beta+2}}+C  |spt(\eta)|^{\frac{1}{\beta+2}}
M^{\frac{p}{2}}L^{\frac{\beta+1}{\beta+2}}(M^{p-2} ||\nabla_0\eta||_{L^\infty}^2+||\eta\partial_t\eta||_{L^\infty})^{\frac 1 2},
\]
from which \eqref{tderivative} follows. This completes the proof.
\end{proof}

\section{Concluding remarks and some open problems}

There are a number of immediate extensions which we want to highlight, as well as some more involved,  plausible extensions which we listen as open problems.

First of all, the prototype for the class of operators in \eqref{maineq-zero}  is the regularized $p$-Laplacian operator $$L_p u= \text{div}_{g_0,\mu_0}( (\ddelta+|\nabla_0 u|_{g_0}^2)^{\frac{p-2}{2}} \nabla_0 u)$$ 
in a sub-Riemannian contact manifold $(M,\omega,g_0)$, where $M$ is the underlying differentiable manifold, $\omega$ is the contact form and $g_0$ is a Riemannian metric on the contact distribution. The measure $\mu_0$ is the corresponding Popp measure.
Since the structure conditions \eqref{structure-zero} and equation \eqref{maineq-zero} are invariant by contact diffeomorphisms, then invoking Darboux coordinates one can pull-back the PDE from the setting of contact subRiemannian manifolds to that of the Heisenberg group. Consequently all our results extends to the more general contact subRiemannian setting.
For a more detailed description, see \cite[Section 6.1]{CCLDO}.  As an immediate corollary of Theorem \ref{main1} one has the following.

\begin{thrm}\label{main3}  Let $(M,\omega,g_0)$ be a contact, sub-Riemannian manifold and let $\Om\subset M$ be an open set. For $2\le p\le 4$,  $\delta\ge 0$,
consider $u\in L^p((0,T), W_0^{1,p}(\Om))$ be a weak solution of 
$$\p_t u =\operatorname{div}_{g_0,\mu_0}( (\ddelta+|\nabla_0 u|_{g_0}^2)^{\frac{p-2}{2}} \nabla_0 u) ,$$ in $Q=\Om\times (0,T)$.
 For any open ball $B\subset \subset \Om$ and $T>t_2\ge t_1\ge 0$,  and $q\ge 1$,
there exist  constants $C=C(n, p, d(B, \p\Om), T-t_2, \ddelta)>0$  and $C_q=C(n, p,q, d(B, \p\Om), T-t_2, \ddelta)>0$ such that
$$
||\nabla_0 u||_{L^\infty(B\times (t_1,t_2))} \le C \text{ and } ||\p_t u||_{L^q(B\times (t_1,t_2))} + ||Z u||_{L^q(B\times (t_1,t_2))} \le C_q.$$
\end{thrm}
Of course, if $\delta>0$ then in view of the results in \cite{CCG}, the solutions are smooth in $Q$.

Some of the following extensions seem   challenging, and we list them as open problems in increasing order of their perceived difficulty.

\begin{enumerate}
\item Standard, but technically involved, modifications should allow to extend our work to the case of equations of the type
$$\p_t u - X_i A_i (x,t, u, \nabla_0 u) = B(x,t,u,\nabla_0 u)$$
with structure conditions similar to those in \cite[Section 1, Chapter VIII]{DB}. 
\item We  feel it should possible to weaken the bounds in the structure conditions for $\p_{x_k} A_i$ and request instead  only horizontal derivatives bounds, bounds  on $X_k A_i$, although this would require some additional work in the proof of Lemma \ref{SobXU}. 
\item This paper only deals with scalar equations, however in the Euclidean case the results continue to hold also for systems of equations with additional structure (see \cite{DB}). The extension in the subelliptic setting would involve first extending the results of \cite{CCG}, and all the regularity theory literature that is used there.
\item Because our argument rests in a crucial way on Lemma \ref{SobXU}, the Lipschitz regularity for the range $4< p <\infty$ is currently beyond our reach. We conjecture that our main Caccioppoli inequality \eqref{maincacci} still holds with exactly the same statement in this extended range.
\item Proof of the  H\"older regularity of horizontal derivatives, in any range of $p\neq 2$.
\item Just as in the Euclidean case, the regularity problem in the range $1<p<2$ is more challenging, and would require completely different arguments. In the stationary case this has been solved by Mukherjee and Zhong in \cite{MZ}.
\item Our work extends easily to any step two Carnot group. Although there is promising work by Domokos and Manfredi \cite{DomMan} in the stationary case, the extension of our result to the higher step setting seems quite challenging.
\end{enumerate}

\end{document}